\documentclass[10pt]{amsart}
\usepackage{amsmath}
\usepackage{amsthm}
\usepackage{amsfonts}
\usepackage{amssymb}
\usepackage{hyperref}

\theoremstyle{plain}
\newtheorem{theorem}{Theorem}
\newtheorem{lemma}{Lemma}

\newtheorem{corollary}{Corollary}

\numberwithin{equation}{section}

\newcommand{\mathmod}[1]{\ \left(\mathrm{mod}\ #1\right)}

\begin{document}
\title{On a theorem of Kanold on odd perfect numbers
\footnotetext{T\MakeLowercase{his preprint is the author's final version of the paper of the same title
published in} \\
\textit{I\MakeLowercase{ndian} J. P\MakeLowercase{ure} A\MakeLowercase{ppl}. M\MakeLowercase{ath}.}, \MakeLowercase{which is available online at \\ \url{https://doi.org/10.1007/s13226-023-00530-y}.}}}
\author{Tomohiro Yamada}
\address{Center for Japanese language and culture, Osaka University,
562-8678, 3-5-10, Sembahigashi, Minoo, Osaka, Japan}
\email{tyamada1093@gmail.com}

\date{}

\begin{abstract}
We shall prove that if $N=p^\alpha q_1^{2\beta_1} q_2^{2\beta_2} \cdots q_{r-1}^{2\beta_{r-1}}$
is an odd perfect number
such that $p, q_1, \ldots, q_{r-1}$ are distinct primes, $p\equiv\alpha\equiv 1\mathmod{4}$
and $t$ divides $2\beta_i+1$ for all $i=1, 2, \ldots, r-1$, then $t^5$ divides $N$,
improving an eighty-year old result of Kanold.
\end{abstract}

\maketitle

\section{Introduction}
A positive integer $N$ is called perfect if the sum of divisors of $N$ except $N$ itself is equal to $N$.
In other words, a perfect number is a positive integer $N$ satisfying $\sigma(N)=2N$,
where $\sigma(N)$ denotes the sum of divisors of $N$ including $N$ itself as usual.
It is one of the oldest problem in mathematics whether or not an odd perfect number exists.
Euler showed that an odd perfect number must be of the form
\begin{equation}\label{eq11}
N=p^{\alpha} q_1^{2\beta_1}\cdots q_{r-1}^{2\beta_{r-1}}
\end{equation}
where $q_1, \ldots, q_{r-1}$ and $p$ are distinct odd primes
and $\beta_1, \ldots, \beta_{r-1}$ and $\alpha$ are positive integers with $p\equiv \alpha\equiv 1\mathmod{4}$.
Nielsen proved that $N<2^{4^r}$ in \cite{Nie1} and then $N<2^{(2^r-1)^2}$ in \cite{Nie2}.

However, we do not know a proof (or disproof) of the nonexistence of odd perfect numbers
even of the special form $\beta_1=\beta_2=\cdots =\beta_{r-1}=\beta$ in \eqref{eq11}
with $\beta$ an integer,
although McDaniel and Hagis conjectured that there exists no such one in \cite{MDH}.
A series of papers studying odd perfect numbers of this form
was started by Steuerwald \cite{St}, who showed that $\beta\neq 1$.
Now it is known that $\beta\neq 2$ from Kanold \cite{Kan1},
$\beta\neq 3$ from Hagis and McDaniel \cite{HMD},
$\beta\neq 5, 12, 24, 62$ from Hagis and McDaniel \cite{MDH},
$\beta\neq 6, 8, 11, 14, 18$ from Hagis and McDaniel \cite{CW},
$\beta\not\equiv 1\mathmod{3}$ from McDaniel \cite{Mc},
and $\beta\not\equiv 2\mathmod{5}$ from Fletcher, Nielsen, and Ochem \cite{FNO}.
The author proved that $r\leq 4\beta^2+2\beta+3$ and $N<2^{4^{4\beta^2+2\beta+3}}$
for an odd perfect number $N$ in the form \eqref{eq11} in \cite{Ymd1}.
Later, the author replaced the upper bound for $r$ by
$2\beta^2+8\beta+3$ in \cite{Ymd2} and then $2\beta^2+6\beta+3$ in \cite{Ymd3}.

Indeed, McDaniel proved a slightly stronger result that
if $2\beta_i+1\equiv 0\mathmod{3}$ for all $i=1, 2, \ldots, r-1$ in \eqref{eq11},
then $N$ can never be an odd perfect number.
Similarly, Fletcher, Nielsen, and Ochem \cite{FNO} proved that
an odd integer $N$ in the form \eqref{eq11} can never be perfect if
$2\beta_i+1\equiv 0\mathmod{5}$ for all $i=1, 2, \ldots, r-1$.

Around eighty years ago, Kanold \cite{Kan1} had already proved a general result that
if $N$ in the form \eqref{eq11} is an odd perfect number and
an integer $t$ divides $2\beta_i+1$ for $i=1, 2, \ldots , r-1$, then $t^4$ divides $N$.
The aim of this paper is to strengthen this result.

\begin{theorem}
If $N$ is an odd perfect number in the form \eqref{eq11} and
an integer $t$ divides $2\beta_i+1$ for $i=1, 2, \ldots, r-1$, then $t^5$ divides $N$.
\end{theorem}

Our tools are well-known divisibility properties of values of cyclotomic polynomials
and a technical lemma on a certain system of exponential diophantine equations.

We observe that it suffices to prove the theorem for $t=\ell^k$ with $\ell$ prime
and $k$ a positive integer.
Indeed, once we have proved it, we see that if $N$ is an odd perfect number in the form \eqref{eq11} and
an integer $t=\ell_1^{k_1} \cdots \ell_s^{k_s}$ with $\ell_1, \ldots, \ell_k$ distinct primes
divides $2\beta_i+1$ for $i=1, 2, \ldots, r-1$, then $\ell_j^{5k_j}$ divides $N$
for each $j=1, 2, \ldots, s$ and therefore $t^5$ divides $N$ as desired.

Hence, we shall only discuss the case $t=\ell^k$ with $\ell$ prime
and $k$ a positive integer.
This case can be divided into several cases according to the value of $k$ and whether $p=\ell$ or not.
Many cases can be easily settled using well-known divisibility properties of values of cyclotomic polynomials.
The only case that requires elementary but numerous steps is the case $p=\ell$ and $k=2$.
Eventually, this case can be settled by analyzing several exponential diophantine equations
involving primes, completing our proof.

\section{Preliminaries}\label{preliminaries}

For a prime $p$ and an integer $e\geq 0$, we write $p^e\mid x$ if $p^e$ divides $x$ and $p^e\mid\mid x$
if $p^e$ divides $x$ but $p^{e+1}$ does not.
Let $\Phi_d(X)$ denote the $d$-th cyclotomic polynomial
and $o_p(x)$ denote the multiplicative order of an integer $x$ modulo $p$ for a prime $p$.

Now we quote some elementary divisibility properties of numbers of the form $\Phi_d(x)$.
Lemmas \ref{lm1} and \ref{lm2} are well-known.
Lemma \ref{lm1} follows from Theorems 94 and 95 in Nagell \cite{Nag}.
Lemma \ref{lm2} has been proved by Bang \cite{Ban}.
A generalization by Zsigmondy \cite{Zsi} is well-known and
rediscovered by many authors such as Dickson \cite{Dic} and Kanold \cite{Kan2}.

\begin{lemma}\label{lm1}
Let $x$ and $d$ be positive integers.
Then a prime $p$ divides $\Phi_d(x)$ if and only if $d=p^e o_p(x)$
for some integer $e\geq 0$.
Furthermore, if $e\geq 1$, then $p\mid\mid\Phi_d(x)$.
In particular, if a prime $p$ divides $\Phi_d(x)$,
then $d=o_p(x)$ or $p$ is the largest prime divisor of $d$
and, in the latter case, $p\mid\mid\Phi_d(x)$.
\end{lemma}

\begin{lemma}\label{lm2}
If $a$ and $d$ are integers greater than $1$, then $\Phi_d(a)$ has
a prime factor $p$ such that $d=o_p(a)$,
which we shall call a \textit{primitive prime factor},
unless $(a, d)=(2, 6)$ or $d=2$ and $a+1$ is a power of $2$.
\end{lemma}

The following corollary immediately follows from Lemma \ref{lm1}.

\begin{corollary}\label{cor1}
If a prime $p$ divides both $\Phi_k(a)$ and $\Phi_\ell(a)$ with $\ell>k$,
then $\ell=p^e k$ for some integer $e\geq 1$
and $p\mid\mid\Phi_\ell(a)$.
\end{corollary}

We use the following technical lemma on a certain system of two exponential diophantine equations
obtained by Kanold \cite{Kan3}.
\begin{lemma}\label{K}
If three primes $\ell, q_1, q_2$ and four positive integers $e_i, f_i (i=1, 2)$ satisfy two relations
$\Phi_\ell(q_1^{e_1})=\ell q_2^{f_1}$ and $\Phi_\ell(q_2^{e_2})=\ell q_1^{f_2}$ simultaneously,
then $\ell=2$ and $(q_1^{e_1}, q_2^{e_2})=(3^2, 5)$ or $(5, 3^2)$.
\end{lemma}

\section{Proof of the theorem}

As observed in the Introduction, it suffices to prove the theorem for $t=\ell^k$ with $\ell$ prime.
To this end, we assume that $N$ is an odd perfect number in the form \eqref{eq11}
and, for a prime $\ell$ and a positive integer $k$, $\ell^k$ divides $2\beta_i+1$ for all $i=1, 2, \ldots, r-1$
but $\ell^{5k}$ does not divide $N$.
We note that $\ell^{4k}$ divides $N$ from Kanold \cite{Kan1}.
Moreover, Kanold's proof yields that $q_i\equiv 1\mathmod{\ell}$ for some $q_i$.

Let $S$ denote the set of prime factors $q_i$ of $N$ such that $q_i\equiv 1\mathmod{\ell}$.
For each $q_i$ in $S$, $\ell$ divides $\Phi_{\ell^j}(q_i)$ for any integer $j\geq 1$
and therefore $\ell^k\mid \prod_{j=1}^k\Phi_{\ell^j}(q_i)\mid \sigma(q_i^{2\beta_i})$.
Thus, we observe that
\begin{equation}\label{eq31}
1\leq \#S\leq 4.
\end{equation}

We begin with settling relatively easy cases.

\begin{lemma}\label{A}
Let $\ell$ be a prime and $k$ be a positive integer.
If $\ell^k$ divides $2\beta_i+1$ for all $i=1, 2, \ldots, r-1$ but $\ell^{5k}$ does not divide $N$,
then (a) $\ell\neq q_i$ for all $i=1, 2, \ldots, r-1$ and (b) $\ell\neq p$ unless $k=2$.
\end{lemma}

\begin{proof}
We divide into five cases: (I) $\ell=q_i$, $k=1$, (II) $\ell=q_i$, $k\geq 2$, (III) $\ell=p$, $k=1$,
(IV) $\ell=p$, $k\geq 4$, and (V) $\ell=p$, $k=3$.

\textbf{Case-I $\ell=q_i$, $k=1$.}
If $k=1$ and $\ell=q_i$ for some $i$, then $\ell^{2\beta_i}$ divides $N$ with $2\beta_i\geq \ell-1$.
If $\ell\geq 7$, then $\ell^6$ divides $N$, which contradicts the assumption.

If $q_i=\ell=3$, then, since $3$ divides $2\beta_i+1$ but $2\beta_i<5k=5$,
we must have $2\beta_i=2$.
We observe that $\sigma(3^2)=13$ divides $N$.
If $p=13$, then $7, 19, 127\in S$ since $(p+1)/2=7$, $\sigma(7^2)=3\times 19$,
and $\sigma(19^2)=3\times 127$.
Similarly, if $p\neq 13$, then we must have either $13, 61, 97\in S$,
$p=61$ and $13, 31, 331\in S$, or $p=97$ and $13, 61, 7\in S$.
Hence, $\#S\geq 3$ and $2\beta_i\geq 3$, which is a contradiction.

If $q_i=\ell=5$, then we must have $2\beta_i=4$ like above.
We observe that $11, 71, 211\in S$ since $\sigma(5^4)=11\times 71$,
$\sigma(11^4)=5\times 3221$,
$\sigma(71^4)=5\times 11\times 211\times 2221$,
and $\sigma(211^4)=5\times 1361\times 292661$.
Moreover, at least two primes $q$ of $1361$, $2221$, $3221$, $292661$
belong to $S$ and therefore $\#S\geq 5$, contrary to \eqref{eq31}.

We note that the impossibilities of $\ell=3$ and $\ell=5$ immediately follow from
McDaniel's result in \cite{Mc} and a result of Fletcher, Nielsen, and Ochem \cite{FNO} mentioned
in the Introduction respectively.
However, in our cases, we need only a small amount of computation compared to their arguments.

\textbf{Case-II $\ell=q_i$, $k\geq 2$.}
If $k\geq 2$ and $\ell=q_i$ for some $i$, then $\ell^{2\beta_i}$ divides $N$
with $2\beta_i\geq \ell^k-1\geq 5^k-1\geq 5k$.
Hence, $\ell^{5k}$ divides $N$, contrary to the assumption.

\textbf{Case-III $\ell=p$, $k=1$.}
If $k=1$, then $p^4$ divides $N$ and therefore $\alpha\geq 4$.
Since $\alpha\equiv 1\mathmod{4}$, $\alpha\geq 5$ and $p^5=\ell^{5k}$ divides $N$,
which is a contradiction.

\textbf{Case-IV $\ell=p$, $k\geq 4$.}
By \eqref{eq31}, we can take a prime factor $q_i\equiv 1\mathmod{\ell}$.
It follows from Lemma \ref{lm2} that each $\Phi_{\ell^j}(q_i) ~ (j=1, 2, 3, 4)$ has a primitive prime factor,
which must belong to $S$ since $p=\ell$.
Thus, $\#S\geq 5$, which contradicts \eqref{eq31}.

\textbf{Case-V $\ell=p$, $k=3$.}
Let $q_i$ be a prime factor $\equiv 1\mathmod{\ell}$.
Like above, each $\Phi_{\ell^j}(q_i) ~ (j=1, 2, 3)$ has a primitive prime factor.
Thus, we observe that $\#S\geq 4$ and, together with \eqref{eq31},
\begin{equation}\label{eq32}
\#S=4.
\end{equation}

Assume that there exists an index $j$ such that $q_j\not\equiv 1\mathmod{\ell}$
but $\ell$ divides $\sigma(q_j^{2\beta_j})$.
Let $d$ be the multiplicative order of $q_j\mathmod{\ell}$.
It follows from Lemma \ref{lm1} that $d\ell^2$ divides $2\beta_j+1$, $\gcd(d, \ell)=1$ and
$\ell$ divides $\Phi_d(q_j)$, $\Phi_{d\ell}(q_j)$, and $\Phi_{d\ell^2}(q_j)$.
Hence, $\ell^3$ divides $\sigma(q_j^{2\beta_j})$ and, together with \eqref{eq32},
we see that $\ell^{5k}$ divides $\prod_{q_i\in S\cup\{q_j\}}\sigma(q_j^{2\beta_j})$
and then $N$, which is a contradiction.
Thus, we see that if $\ell$ divides $\sigma(q_i^{2\beta_i})$, then $q_i\in S$.

If $\ell^3\mid\mid(2\beta_i+1)$ for all $q_i$ in $S$,
then $\ell^{12}\mid\mid N$, which is impossible since $\alpha\equiv 1\mathmod{4}$.
Hence, there must exist a prime factor $q_j$ in $S$ such that $\ell^4$ divides $\sigma(q_j^{2\beta_j})$.
Lemma \ref{lm1} yields that $\ell^4$ divides $2\beta_j+1$.
each $\Phi_{\ell^i}(q_j) ~ (i=1, 2, 3, 4)$ has a primitive prime factor.
Thus, $\#S\geq 5$, which contradicts \eqref{eq32}.
\end{proof}

Now the only case that remains is the case $\ell=p$ and $k=2$,
which is the most difficult case.

Since we have assumed that $\ell^{5k}$ does not divide $N$, we see that
$8=4k\leq \alpha\leq 5k-1=9$ and therefore we must have $\alpha=9$.
This will eventually turn out to be impossible after several steps.

By \eqref{eq31}, we can take a prime factor $q_i$ in $S$
and each $\Phi_{\ell^j}(q_i) ~ (j=1, 2)$ has a primitive prime factor.
Thus, $N$ has at least three prime factors $q_i\equiv 1\mathmod{\ell}$.
The following three lemmas yield that $S$ consists exactly four primes and $2\beta_i+1=\ell^3$
for exactly one prime $q_i$ in $S$ and $2\beta_i+1=\ell^2$ for the other three in $S$.

\begin{lemma}\label{b}
$N$ has no prime factor $q_j$ of $N$ such that
the multiplicative order $d$ of $q_j\mathmod{\ell}$ divides $2\beta_j+1$ and $d>1$.
In other words, if $\ell$ divides $\sigma(q_i^{2\beta_i})$, then $q_i$ must belong to $S$.
\end{lemma}

\begin{proof}
Assume that $N$ has a prime factor $q_j$ of $N$ such that
the multiplicative order $d$ of $q_j\mathmod{\ell}$ divides $2\beta_j+1$ and $d>1$.
Then, each of $\Phi_{\ell}(q_j)$, $\Phi_{\ell^2}(q_j)$, $\Phi_{d\ell}(q_j)$, $\Phi_{d\ell^2}(q_j)$
has a primitive prime factor.
Hence, $N$ must have at least four prime factors $q_i\equiv 1\mathmod{\ell}$.
Now we see that $\ell^8$ must divide $\prod_{q_i\in S}\sigma(q_i^{2\beta})$ and
\begin{equation}
\ell^{10}\mid\Phi_{d\ell}(q_j)\Phi_{d\ell^2}(q_j)\prod_{q_i\in S}\sigma(q_i^{2\beta})
\mid\prod_{q_i\in S\cup\{q_j\}}\sigma(q_i^{2\beta}),
\end{equation}
which is a contradiction.
Hence, $N$ has no prime factor $q_j$ of $N$ such that the multiplicative order $d$ of $q_j\mathmod{\ell}$
divides $2\beta_j+1$.
\end{proof}

\begin{lemma}\label{c}
$N$ has exactly four prime factors $q_i\equiv 1\mathmod{\ell}$.
\end{lemma}

\begin{proof}
Otherwise, we have $\#S\leq 3$ and, 
since $\ell^9$ divides $\prod_{q_i\in S}\sigma(q_i^{2\beta_i})$ from Lemma \ref{b},
we observe that $\ell^3$ divides some $\sigma(q_i^{2\beta_i})$ with $q_i\equiv 1\mathmod{\ell}$.
Thus, $\ell^3$ divides $2\beta_i+1$ and
$\prod_{j=1}^3 \Phi_{\ell^j}(q_i)\mid \sigma(q_i^{2\beta_i})\mid N$.
Since each $\Phi_{\ell^j}(q_i)$ has a primitive prime factor,
$N$ has at least three prime factors $\equiv 1\mathmod{\ell}$ other than $q_i$
and therefore $\#S\geq 4$.
This is a contradiction.
\end{proof}

\begin{lemma}\label{d}
$2\beta_i+1=\ell^3$ for exactly one $q_i$ in $S$
and $2\beta_i+1=\ell^2$ for other three $q_i$'s in $S$.
\end{lemma}

\begin{proof}
From Lemmas \ref{b} and \ref{c}, we see that $N$ must have exactly four prime factors $q_i\equiv 1\mathmod{\ell}$
but no prime factor $q_j$ of $N$ such that the multiplicative order $d$ of $q_j\mathmod{\ell}$
divides $2\beta_j+1$ and $d>1$.
This means that $\ell^9\mid\mid\prod_{q_i\in S}\sigma(q_i^{2\beta_i})$ with $\# S\leq 4$.
Hence, $\ell^3$ divides some $\sigma(q_i^{2\beta_i})$ with $q_i\equiv 1\mathmod{\ell}$.

But, since $\ell^2$ divides $2\beta_i+1$ for any $i=1, 2, \ldots, r-1$,
$\ell^2$ must divide $\sigma(q_i^{2\beta_i})$ for each $q_i$ in $S$.
Since $\ell^9\mid\mid\prod_{q_i\in S}\sigma(q_i^{2\beta_i})$,
we must have $\ell^3\mid\mid\sigma(q_i^{2\beta_i})$ for exactly one $i$ in $S$
and $\ell^2\mid\mid\sigma(q_i^{2\beta_i})$ for other three $i$'s in $S$.

Now Lemma \ref{lm1} yields that $\ell^3\mid\mid(2\beta_i+1)$ for exactly one $q_i$'s in $S$
and $\ell^2\mid\mid(2\beta_i+1)$ for other three $q_i$'s in $S$.

If a prime $s$ other than $\ell$ divides $2\beta_i+1$ for some $q_i$ in $S$,
then, each of $\Phi_{\ell}(q_i)$, $\Phi_{\ell^2}(q_i)$, $\Phi_{s\ell}(q_i)$, $\Phi_{s\ell^2}(q_i)$
has a primitive prime factor and we must have $\#S\geq 5$, contrary to \eqref{eq31}.
Thus the lemma holds.
\end{proof}

Now we shall show that several diophantine relations involving primes in $S$ must hold.
We renumber indices so that $S=\{q_1, q_2, q_3, q_4\}$,
$2\beta_i+1=\ell^2$ for $i=1, 2, 3$ and $2\beta_4+1=\ell^3$.
Thus, $\sigma(q_4^{2\beta_4})=\prod_{j=1}^3\Phi_{\ell^j}(q_4)$ must be composed of
only four primes $q_1, q_2, q_3$ and $\ell$.
We can deduce from Corollary \ref{cor1} that $\gcd(\Phi_{\ell^i}(q_4), \Phi_{\ell^j}(q_4))=\ell$
if $i\neq j$.
Hence, we have
\begin{equation}\label{eq33}
\Phi_{\ell}(q_4)=\ell q_a^{h_a}, \Phi_{\ell^2}(q_4)=\ell q_b^{h_b}, \Phi_{\ell^3}(q_4)=\ell q_c^{h_c}
\end{equation}
for some permutation $(a, b, c)$ of $(1, 2, 3)$ and positive integers $h_a, h_b, h_c$.
We can sort $q_i (i=1, 2, 3)$ so that $\Phi_{\ell^3}(q_4)=\ell q_1^{h_1}$.

\begin{lemma}\label{e}
We must have
\begin{equation}\label{eq34}
\Phi_{\ell}(q_1)=\ell q_u^{e_u} q_4^{e_4}, \Phi_{\ell^2}(q_1)=\ell q_v^{e_v}
\end{equation}
or
\begin{equation}\label{eq35}
\Phi_{\ell}(q_1)=\ell q_u^{e_u}, \Phi_{\ell^2}(q_1)=\ell q_v^{e_v} q_4^{e_4}.
\end{equation}
where $(u, v)=(2, 3)$ or $(3, 2)$ and $e_2, e_3, e_4\geq 1$.
\end{lemma}

\begin{proof}
We can easily see that $\Phi_{\ell}(q_1)\Phi_{\ell^2}(q_1)=\ell^2 \prod_{2\leq j\leq 4, j\neq i}q_j^{e_j}$.
By Lemma \ref{K},
we cannot have $\Phi_{\ell}(q_1)=\ell q_4^{e_4}$ or $\Phi_{\ell^2}(q_1)=\ell q_4^{e_4}$.
Hence, we must have \eqref{eq34} or \eqref{eq35}
where $(u, v)=(2, 3)$ or $(3, 2)$ and $e_2, e_3\geq 1$.

Now all that remains is to show that $e_4\geq 1$.
To this end, assume that $e_4=0$ to the contrary.
Now we may assume that $\Phi_{\ell}(q_1)=\ell q_2^{e_2}, \Phi_{\ell^2}(q_u)=\ell q_3^{e_3}$
without the loss of generality.
By Lemma \ref{K}, neither $\Phi_{\ell}(q_2)=\ell q_1^{f_1}, \ell q_4^{f_4}$ nor $\Phi_{\ell^2}(q_2)=\ell q_1^{f_1}, \ell q_4^{f_4}$ holds.
However, $\Phi_{\ell}(q_2)\Phi_{\ell^2}(q_2)$ is composed of only four primes $q_1, q_3, q_4$ and $\ell$
and, $\gcd(\Phi_{\ell}(q_2), \Phi_{\ell^2}(q_2))=\ell$ by Corollary \ref{cor1}.
Hence, we must have $\Phi_{\ell^\eta}(q_2)=\ell q_3^{f_3}$ for some integers $\eta=1$ or $2$ and $f_3\geq 0$.
Similarly, $\Phi_{\ell^\xi}(q_3)=\ell q_2^{g_2}$ for some integers $\xi=1$ or $2$ and $g_2\geq 0$.
This contradicts Lemma \ref{K}.
Thus, we must have $e_4\geq 1$.
\end{proof}

Now we can renumber indices so that
\begin{equation}\label{eq36}
\Phi_{\ell^\gamma}(q_1)=\ell q_3^{e_3}, \Phi_{\ell^{3-\gamma}}(q_1)=\ell q_2^{e_2} q_4^{e_4}
\end{equation}
with $e_2, e_3, e_4\geq 1$ and $\gamma=1$ or $2$.

\begin{lemma}
$\ell$ divides $e_4$.
\end{lemma}

\begin{proof}
It follows from \eqref{eq36} that
$q_1^{\ell(\ell-1)}+\cdots +q_1^\ell+1=\ell q_2^{e_2} q_4^{e_4}$
or $q_1^{\ell-1}+\cdots +q_1+1=\ell q_2^{e_2} q_4^{e_4}$.
Hence, we obtain
\begin{equation}\label{eq37}
\ell q_2^{e_2} q_4^{e_4}\equiv 1\mathmod{q_1}.
\end{equation}

We see that $q_3^{\ell^2}\equiv 1\mathmod{q_1}$
since otherwise we must have $\Phi_{\ell}(q_3)=\ell q_4^{f_4}$ or $\Phi_{\ell^2}(q_3)=\ell q_4^{f_4}$,
which is incompatible with \eqref{eq33} by Lemma \ref{K}.
Hence, we obtain
\begin{equation}\label{eq38}
\ell^{\ell^2}\equiv \ell^{\ell^2} q_3^{e_3\ell^2}\equiv 1\mathmod{q_1}.
\end{equation}
Similarly, we see that $q_2^{\ell^2}\equiv 1\mathmod{q_1}$ and, combining \eqref{eq37} and \eqref{eq38},
we obtain
\begin{equation}
q_4^{e_4\ell^2}\equiv (\ell q_2^{e_2} q_4^{e_4})^{\ell^2}\equiv 1\mathmod{q_1}.
\end{equation}
But, since $q_1$ divides $\Phi_{\ell^3}(q_4)$, $\ell$ must divide $e_4$.
\end{proof}

\begin{lemma}\label{g}
$\Phi_\ell(q_1)=\ell q_3^{e_3}$ and $\Phi_{\ell^2}(q_1)=\ell q_2^{e_2} q_4^{e_4}$.
\end{lemma}

\begin{proof}
We begin by observing that
\begin{equation}
\ell q_1^{e_1}\equiv 1\mathmod{q_4^{\ell^2}}
\end{equation}
from the assumption that $\Phi_{\ell^3}(q_4)=\ell q_1^{h_1}$.
Hence, if $\Phi_\ell(q_1)=\ell q_2^{e_2} q_4^{e_4}$, then we must have
\begin{equation}
\ell^\ell\equiv (\ell q_1^{e_1})^\ell\equiv 1\mathmod{q_4^{\min\{e_4, \ell^2\}}}.
\end{equation}
Since $e_4\geq \ell$ by the previous lemma, 
$q_4^\ell$ must divide $\ell^\ell-1$.
But, since $q_4\in S$, we have $q_4>\ell$ and $q_4^\ell>\ell^\ell-1$.
This is a contradiction.
\end{proof}

\begin{lemma}\label{h}
$q_4^\ell$ divides $\Phi_{\ell^2}(\ell)$ and $q_4\equiv 1\mathmod{\ell^2}$.
\end{lemma}

\begin{proof}
Since we have just shown that $\Phi_{\ell^2}(q_1)=\ell q_2^{e_2} q_4^{e_4}$,
we obtain
\begin{equation}
\ell^{\ell^2}\equiv (\ell q_1^{h_1})^{\ell^2}\equiv 1\mathmod{q_4^{\ell}},
\end{equation}
where we used that $\min\{e_4, \ell^2\}\geq\ell$ as in the previous lemma.
By Lemma \ref{lm2}, $\gcd(\ell^\ell-1, \Phi_{\ell^2}(\ell))=1$
and therefore $q_4^\ell$ divides $\ell^\ell-1$ or $\Phi_{\ell^2}(\ell)$.
But, since $q_4>\ell$, $q_4^\ell$ cannot divide $\ell^\ell-1$.
Hence, $q_4^\ell$ divides $\Phi_{\ell^2}(\ell)$ and Lemma \ref{lm1} yields that $q_4\equiv 1\mathmod{\ell^2}$.
\end{proof}

\begin{lemma}\label{i}
$\Phi_\ell(q_2)=\ell q_1^{f_1}$.
\end{lemma}

\begin{proof}
From \eqref{eq33}, we see that
\begin{equation}
\ell q_2^{h_2}\equiv 1\mathmod{q_4^\ell}.
\end{equation}
If $q_4$ divides $\Phi_\ell(q_2)$, then,
\begin{equation}
\ell^\ell\equiv (\ell q_2^{h_2})^\ell\equiv 1\mathmod{q_4}.
\end{equation}
Hence, $q_4$ must divide $\ell^\ell-1$, contrary to the previous lemma.

Thus, we must have $\Phi_\ell(q_2)=\ell q_1^{f_1} q_3^{f_3}$.
If both $f_1, f_3\geq 1$, then we must have $\Phi_{\ell^2}(q_2)=\ell q_4^{f_4}$,
which is impossible by Lemma \ref{K}.
If $\Phi_\ell(q_2)=\ell q_3^{f_3}$, then,
by Lemma \ref{K}, we can have neither $\Phi_{\ell^s}(q_3)=\ell q_1^{g_1}$, $\ell q_2^{g_2}$ nor $\ell q_4^{g_4}$
for $s=1, 2$, which is a contradiction.
Thus we must have $\Phi_\ell(q_2)=\ell q_1^{f_1}$.
\end{proof}

\begin{lemma}\label{j}
$h_1\geq \ell$.
\end{lemma}

\begin{proof}
Assume that $h_1\leq \ell-1$.
Then, we can easily see that
\begin{equation}
q_1>\left(\frac{q_4^{\ell^2(\ell-1)}}{\ell}\right)^{1/(\ell-1)}>\frac{q_4^{\ell^2}}{2}.
\end{equation}
By Lemma \ref{g},
we must have $\Phi_{\ell^2}(q_1)=\ell q_2^{e_2} q_4^{e_4}$.
Hence, observing that $\ell<2^{\ell-1}$ and $\ell(\ell-1)>(\ell+1)(\ell-2)$, we obtain
\begin{equation}
q_2^{e_2}>\frac{q_1^{\ell(\ell-1)}}{\ell q_4^{e_4}}\geq \frac{(q_4^{\ell^2}/2)^{\ell(\ell-1)}}{\ell q_4^{\ell^2(\ell-1)}}=\frac{q_4^{\ell^2(\ell-1)^2}}{2^{\ell(\ell-1)}\ell}
>\frac{q_4^{\ell(\ell-1)(\ell+1)(\ell-2)}}{2^{\ell^2-1}}
\end{equation}
and, noting that $e_2\leq 2\beta_2=\ell^2-1$,
\begin{equation}
q_2>\frac{q_4^{\ell(\ell-2)}}{2}.
\end{equation}

Now, Lemma \ref{i} implies that
\begin{equation}
\Phi_{\ell^2}(q_2)=\ell q_3^{f_3} q_4^{f_4}
\end{equation}
for some nonnegative integers $f_3\leq 2\beta_3=\ell^2-1$ and $f_4\leq 2\beta_4=\ell^3-1$.
Hence,
\begin{equation}
q_3^{f_3}>\frac{q_2^{\ell(\ell-1)}}{\ell q_4^{e_4}}\geq \frac{(q_4^{\ell(\ell-2)}/2)^{\ell(\ell-1)}}{\ell q_4^{\ell^2(\ell-1)}}=\frac{q_4^{\ell^2(\ell-1)(\ell-3)}}{2^{\ell(\ell-1)}\ell}
>\frac{q_4^{\ell(\ell-1)(\ell+1)(\ell-4)}}{2^{\ell^2-1}}
\end{equation}
and
\begin{equation}
q_3>\frac{q_4^{\ell(\ell-4)}}{2}.
\end{equation}

But, since $\Phi_\ell(q_4)=\ell q_2^{h_2}$ or $\ell q_3^{h_3}$ by \eqref{eq33}, we have
\begin{equation}
2q_4^{\ell-1}>\Phi_\ell(q_4)>\frac{\ell q_4^{\ell(\ell-4)}}{2}\geq \frac{\ell q_4^{\ell}}{2}
\end{equation}
and $q_4<4/\ell<1$, which is impossible.
\end{proof}

Now we shall complete the proof of our theorem.
If $f_1\leq \ell-4$, then
\begin{equation}
q_2\leq \ell^{1/(\ell-1)}q_1^{(\ell-4)/(\ell-1)}
\end{equation}
and
\begin{equation}\label{eq311}
q_4^{e_4}=\frac{\Phi_{\ell^2}(q_1)}{\ell q_2^{e_2}}>\frac{q_1^{\ell(\ell-1)-(\ell-4)(\ell+1)}}{\ell^{\ell+2}}=\frac{q_1^{2\ell+4}}{\ell^{\ell+2}}>q_1^{\ell+2}.
\end{equation}
Since 
\begin{equation}
q_4^{\ell^2(\ell-1)}<\Phi_{\ell^3}(q_4)\leq \ell q_1^{\ell^2-1},
\end{equation}
we have
\begin{equation}\label{eq312}
q_1>\frac{q_4^{\ell^2/(\ell+1)}}{\ell^{1/(\ell^2-1)}}>q_4^{\ell-1}.
\end{equation}
Combining \eqref{eq311} and \eqref{eq312}, we obtain
\begin{equation}
q_4^{e_4}>q_1^{\ell+2}>q_4^{(\ell-1)(\ell+2)}>q_4^{\ell^2}.
\end{equation}
It follows from Lemma \ref{g} that
$q_4^{\ell^2}\mid\Phi_{\ell^2}(q_1)\mid (q_1^{\ell^2}-1)$ and therefore
\begin{equation}
\ell^{\ell^2}\equiv (\ell q_1^{h_1})^{\ell^2}\equiv 1\mathmod{q_4^{\ell^2}}.
\end{equation}
Hence, we must have $q_4<\ell$, which is a contradiction.

If $\ell\geq 13$ and $f_1\geq \ell-3$, then, since $q_2$ divides $\Phi_{\ell^2}(q_1)$,
we have
\begin{equation}
\ell^{\ell^2}\equiv (\ell q_2^{e_2} q_4^{e_4})^{\ell^2}\equiv 1\mathmod{q_1^{\min\{f_1, h_1\}}}
\end{equation}
and, by Lemma \ref{j}, $q_1^{\ell-3}<\ell^{\ell^2}$.
We observe that
\begin{equation}
q_4^{\ell^2(\ell-1)}<\Phi_{\ell^3}(q_4)=\ell q_1^{h_1}\leq \ell q_1^{2\beta_1}=\ell q_1^{\ell^2-1}
\end{equation}
and therefore
\begin{equation}
q_4<\ell^{\frac{1}{\ell^2(\ell-1)}} q_1^{\frac{\ell^2}{\ell+1}}
\leq\ell^{\frac{1}{\ell^2(\ell-1)}+\frac{\ell^2}{(\ell+1)(\ell-3)}}.
\end{equation}
Since $\ell\geq 13$, we must have $q_4<\ell^{1+15/154+1/2028}<2\ell$ and $q_4\not\equiv 1\mathmod{\ell}$,
which is a contradiction again.

Now the only possibility left is $p=\ell=5$.
Then $(5+1)/2=3$ and $\Phi_5(3)\Phi_{25}(3)=11^2\times 8951\times 391151$ must divide $N$.
Hence, $3001\times 3221\mid \Phi_5(11)\Phi_{25}(11)\mid N$ and therefore
$S$ must contain at least five primes $11$, $3001$, $3221$, $8951$, and $391151$, contrary to \eqref{eq31}.
This completes the proof of our theorem.

{}
\end{document}